%% file: AMEML2.tex
\documentclass[graybox]{svmult}

\usepackage{mathptmx}       
\usepackage{helvet}         
\usepackage{courier}        
\usepackage{type1cm}        
\usepackage{graphicx}        
\usepackage{multicol}        
\usepackage[bottom]{footmisc}
\usepackage{amsmath,amsfonts,latexsym,amssymb}
\usepackage{graphicx}
\input{macros}

\begin{document}

\title*{Regularization with Approximated $L^2$ Maximum Entropy Method}
\author{J-M. Loubes and P. Rochet}
\institute{J-M. Loubes \at Institut de Math\'ematiques de Toulouse, UMR 5219, Universit\'e Toulouse 3, 118 route de Narbonne
F-31062 Toulouse Cedex 9, France \email{loubes@math.univ-toulouse.fr}
\and P. rochet \at  Institut de Math\'ematiques de Toulouse, UMR 5219, Universit\'e Toulouse 3, 118 route de Narbonne
F-31062 Toulouse Cedex 9, France \email{rochet@math.univ-toulouse.fr}}
%
%
\maketitle

\abstract{We tackle the inverse problem of reconstructing an unknown finite measure $\mu$ from a noisy observation of a generalized moment of $\mu$ defined as the integral of a continuous and bounded operator $\Phi$ with respect to $\mu$. When only a quadratic approximation $\Phi_m$ of the operator  is known, we introduce the $L^2$ approximate maximum entropy solution as a minimizer of a convex functional subject to a sequence of convex constraints. Under several assumptions on the convex functional, the convergence of the approximate solution is established and rates of convergence are provided.}

\section{Introduction}
\label{sec:1}
A number of inverse problems may be stated in the form of reconstructing an unknown measure $\mu$ from observations of generalized moments of $\mu$, i.e., moments $y$ of the form
$$y=\int_{\mathcal{X}}\Phi(x)d\mu(x),$$
where $\Phi:\mathcal{X}\to\mathbb{R}^k$ is a given map.
Such problems are encountered in various fields of sciences, like medical imaging, time-series analysis, speech processing, image restoration from a blurred version of the image, spectroscopy, geophysical sciences, crytallography, and tomography; see for example Decarreau et al (1992), Gzyl (2002), Hermann and Noll (2000), and Skilling (1988).
Recovering the unknown measure $\mu$ is generally an ill-posed problem, which turns out to be difficult to solve in the presence of noise, i.e., one observes $y^{obs}$ given by
\begin{equation}
\label{base}
y^{obs}=\int_{\mathcal{X}}\Phi(x)d\mu(x)+\varepsilon.
\end{equation}
For inverse problems with known operator $\Phi$, regularization techniques allow the solution to be stabilized by giving favor to those solutions which minimize a regularizing functional $J$, i.e., one minimizes $J(\mu)$ over $\mu$ subject to the constraint that $\int_{\mathcal{X}}\Phi(x)d\mu(x) = y$ when $y$ is observed, or $\int_{\mathcal{X}}\Phi(x)d\mu(x) \in K_Y$ in the presence of noise, for some convex set $K_Y$ containing $y^{obs}$.
Several types of regularizing functionals have been introduced in the literature.
In this general setting, the inversion procedure is deterministic, i.e., the noise distribution is not used in the definition of the regularized solution.
Bayesian approaches to inverse problem allow one to handle the noise distribution, provided it is known, yet in general, a distribution like the normal distribution is postulated (see Evans and Stark, 2002 for a survey).
However in many real-world inverse problems, the noise distribution is unknown, and only the output $y$ is easily observable, contrary to the input to the operator.
Consequently very few paired data is available to reliably estimate the noise distribution, thereby causing robustness deficiencies on the retrieved parameters.
Nonetheless, even if the noise distribution is unavailabe to the practitioner, she often knows the {\it noise level}, i.e., the maximal magnitude of the disturbance term, say $\rho>0$, and this information may be reflected by taking a constraint set $K_Y$ of diameter $2\rho$.\\

As an alternative to standard regularizations such as Tikhonov or Galerkin, see for instance Engl, Hanke and Neubauer (1996), we focus on  a regularization functional with grounding in information theory, generally expressed as a negative entropy, leading to {\it maximum entropy} solutions to the inverse problem. In a deterministic framework, maximum entropy solutions have been studied in Borwein and Lewis (1993, 1996), while some others study exist in a Bayesian setting (Gamboa, 1999; Gamboa and Gassiat, 1997), in seismic tomography (Fermin, Loubes and Lude\~na, 2006), in image analysis (Gzyl and Zeev, 2003; Skilling and Gull, 2001).
Regularization with maximum entropy also provides one with a very simple and natural manner to incorporate constraints on the support and the range of the solution (see e.g. the discussion in Gamboa and Gassiat, 1997).\\

In many actual situations, however, the map $\Phi$ is unknown and only an approximation to it is available, say $\Phi_m$,  which converges  in quadratic norm to $\Phi$ as m goes to infinity. In this paper,  following lines devised in Gamboa (1999) and Gamboa and Gassiat (1999) and Loubes and Pelletier (2008), we introduce an approximate maximum entropy on the mean (AMEM) estimate $\hat{\mu}_{m,n}$ of the measure $\mu_X$ to be reconstructed.
This estimate is expressed in the form of a discrete measure concentrated on $n$ points of $\mathcal{X}$.
In our main result, we prove that $\hat{\mu}_{m,n}$ converges to the solution of the initial inverse problem as $m\to\infty$ and $n\to\infty$ and provide a rate of convergence for this estimate.

The paper is organized as follows.
Section 2 introduces some notation and the definition of the AMEM estimate.
In Section 3, we state our main result (Theorem~\ref{th:main}).
Section 4 is devoted to the proofs of our results.

\section{Notation and definitions}
\label{sec:2}
\subsection{Problem position}
\label{subsec:2}
Let $\Phi$ be a continuous and bounded map defined on a subset $\mathcal{X}$ of $\mathbb{R}^d$ and taking values in $\mathbb{R}^k$.
The set of finite measures on $(\mathcal{X},\mathcal{B}(\mathcal{X}))$ will be denoted by $\mathcal{M}(\mathcal{X})$, where $\mathcal{B}(\mathcal{X})$ denotes the Borel $\sigma$-field of $\mathcal{X}$.
Let $\mu_X\in\mathcal{M}(\mathcal{X})$ be an unknown finite measure on $\mathcal{X}$ and consider the following equation:
\begin{equation} \label{pbinv}
y = \int_{\mathcal{X}}\Phi(x)d\mu_X(x).
\end{equation}
Suppose that we observe a perturbed version $y^{obs}$ of the response $y$:
$$y^{obs} = \int_{\mathcal{X}}\Phi(x)d\mu_X(x) + \varepsilon,$$
where $\varepsilon$ is an error term supposed bounded in norm from above by some positive constant $\eta$, representing the maximal noise level.
Based on the data $y^{obs}$, we aim at reconstructing the measure $\mu_X$ with a maximum entropy procedure. As explained in the introduction, the true map $\Phi$ is  {\it unknown} and we assume knowledge of an {\it approximating sequence} $\Phi_m$ to the map $\Phi$, such that
$$\Vert \Phi_m - \Phi \Vert_{\mathbb{L}^2(P_X)} = \sqrt{\mathbb{E}(\Vert \Phi_m(X) - \Phi(X) \Vert^2)} \rightarrow 0, $$
at a rate $\varphi_m$.\\
\indent Let us first introduce some notation.
For all probability measure $\nu$ on $\mathbb{R}^n$, we shall denote by $\mathcal{L}_\nu$, $\Lambda_\nu$, and $\Lambda_\nu^*$ the Laplace, log-Laplace, and Cramer transforms of $\nu$, respectively defined for all $s\in\mathbb{R}^n$ by:
\begin{eqnarray*}
\mathcal{L}_\nu(s) & = & \int_{\mathbb{R}^n}\exp\langle s,x\rangle d\nu(x),\\
\Lambda_\nu(s) & = & \log\mathcal{L}_\nu(s),\\
\Lambda_\nu^*(s) & = & \sup_{u \in \mathbb{R}^n}\{\langle s,u\rangle - \Lambda_\nu(u)\}.
\end{eqnarray*}

Define the set 
$$K_Y = \{y\in\mathbb{R}^k : \|y-y^{obs}\|\leq \eta\},$$
i.e., $K_Y$ is the closed ball centered at the observation $y^{obs}$ and of radius $\eta$.\\

Let $\mathcal{X}$ be a set, and let $\mathcal{P}(\mathcal{X})$ be the set of probability measures on $\mathcal{X}$.
For $\nu,\mu\in\mathcal{P}(\mathcal{X})$, the relative entropy of $\nu$ with respect to $\mu$ is defined by
$$
H(\nu|\mu) =\left\{
\begin{array}{ll}
\int_{\mathcal{X}}\log\left(\frac{d\nu}{d\mu}\right)d\nu & \mathrm{if  } \nu<<\mu\\
 +\infty & \mathrm{otherwise.}
\end{array}\right.
$$
Given a set $\mathcal{C}\in\mathcal{P}(\mathcal{X})$ and a probability measure $\mu\in\mathcal{P}(\mathcal{X})$, an element $\mu^\star$ of $\mathcal{C}$ is called an {\it I-projection} of $\mu$ on $\mathcal{C}$ if
$$H(\mu^\star|\mu) = \inf_{\nu \in \mathcal{C}} H(\nu|\mu).$$

Now we let $\mathcal{X}$ be a locally convex topological vector space of finite dimension.
The dual of $\mathcal{X}$ will be denoted by $\mathcal{X}'$.
The following two Theorems, due to Csiszar (1984), characterize the entropic projection of a given probability measure on a convex set.
For their proofs, see Theorem 3 and Lemma 3.3 in Csiszar (1984), respectively.

\begin{theorem}
\label{theo:csiszar1}
Let $\mu$ be a probability measure on $\mathcal{X}$.
Let $\mathcal{C}$ be a convex subset of $\mathcal{X}$ whose interior has a non-empty intersection with the convex hull of the support of $\mu$.
Let
$$\Pi\left(\mathcal{X}\right) = \{P\in\mathcal{P}(\mathcal{X}) : \int_\mathcal{X} x dP(x) \in \mathcal{C}\}.$$
Then the I-projection $\mu^\star$ of $\mu$ on $\Pi(\mathcal{C})$ is given by the relation
$$d\mu^\star(x) = \frac{\exp\lambda^\star(x)}{\int_\mathcal{X}\exp\lambda^\star(u)d\mu(u)}d\mu(x),$$
where $\lambda^\star \in \mathcal{X}'$ is given by
$$\lambda^\star = \arg\max_{\lambda\in\mathcal{X}'}\left[\inf_{x\in\mathcal{C}}\lambda(x) - \log\int_\mathcal{X}\exp\lambda(x)d\mu(x)\right].$$
\end{theorem}

Now let $\nu_Z$ be a probability measure on $\mathbb{R}_+$.
Let $P_X$ be a probability measure on $\mathcal{X}$ having full support, and define the convex functional $I_{\nu_Z}(\mu|P_X)$ by:
$$
I_{\nu_Z}(\mu | P_X)  =
\begin{cases}
\int_{\mathcal{X}} \Lambda_{\nu_Z}^*\left(\frac{d\mu}{dP_X}\right)dP_X & \mathrm{if}\,\,\mu<<P_X\\
+\infty & \mathrm{otherwise.}
\end{cases}
$$
Within this framework, we consider as a solution of the inverse problem \eqref{pbinv} a minimizer of the functional $I_{\nu_Z}(\mu | P_X)$ subject to the constraint
$$\mu \in S(K_Y) = \{\mu\in\mathcal{M}(\mathcal{X}) : \int_{\mathcal{X}} \Phi(x)d\mu(x) \in K_Y \}.$$

\subsubsection{The AMEM estimate}
\label{subsec:3}
We introduce the approximate maximum entropy on the mean (AMEM) estimate as a sequence $\hat{\mu}_{m,n}$ of discrete measures on $\mathcal{X}$.
In all of the following, the integer $m$ indexes the approximating sequence $\Phi_m$ to $\Phi$, while the integer $n$ indexes a random discretization of the space $\mathcal{X}$.
For the construction of the AMEM estimate, we proceed as follows.\\

Let $(X_1,\dots,X_n)$ be an i.i.d sample drawn from $P_X$.
Thus the empirical measure $\frac{1}{n}\sum_{i=1}^n\delta_{X_i}$ converges weakly to $P_X$.\\

Let $L_n$ be the discrete measure with random weights defined by
$$L_n = \frac{1}{n}\sum_{i=1}^nZ_i\delta_{X_i},$$
where $(Z_i)_i$ is a sequence of i.i.d. random variables on $\mathbb{R}$.\\

For $\mathcal{S}$ a set we denote by $\mathrm{co}\, \mathcal{S}$ its convex hull.
Let $\Omega_{m,n}$ be the probability event defined by
\begin{equation}
\label{eq:omegan}
\Omega_{m,n} = [K_Y \cap \mathrm{co}\, Supp\, F_*\nu_Z^{\otimes n} \neq \emptyset]
\end{equation}
where
$F:\mathbb{R}^n\to\mathbb{R}^k$ is the linear operator associated with the matrix
$\mathbf{A}_{m,n} = \frac{1}{n}(\Phi_m^{i}(X_j))_{(i,j)\in  [1,k]\times [1,n]}$
and where $F_*\nu_Z^{\otimes n}$ denotes the image measure of $\nu_Z^{\otimes n}$ by $F$.
For ease of notation, the dependence of $F$ on $m$ and $n$ will not be explicitely written throughout.\\

Denote by $\mathcal{P}(\mathbb{R}^n)$ the set of probability measures on $\mathbb{R}^n$.
For any map $\Psi:\mathcal{X} \to \mathbb{R}^k$ define the set
$$\Pi_n(\Psi,K_Y) = \left\{\nu\in\mathcal{P}(\mathbb{R}^n) : \mathbb{E}_\nu\left[ \int_{\mathcal{X}} \Psi(x) dL_n(x) \right] \in K_Y \right\}.$$
Let $\nu_{m,n}^\star$ be the I-projection of $\nu_Z^{\otimes n}$ on $\Pi_n(\Phi_m,K_Y)$.\\

Then, on the event $\Omega_{m,n}$, we define the AMEM estimate $\hat{\mu}_{m,n}$  by
\begin{equation}
\hat{\mu}_{m,n} = \mathbb{E}_{\nu_{m,n}^\star}\left[L_n\right],
\end{equation}
and we extend the definition of $\hat{\mu}_{m,n}$ to the whole probability space by setting it to the null measure on the complement $\Omega_{m,n}^c$ of $\Omega_{m,n}$.
In other words, letting $(z_1,...,z_n)$ be the expectation of the measure $\nu_{m,n}^\star$, the AMEM estimate may be rewritten more conveniently as
\begin{equation}
\label{eq:cmemonomega}
\hat{\mu}_{m,n} = \frac{1}{n}\sum_{i=1}^nz_i\delta_{X_i}
\end{equation}
with $z_i=\mathbb{E}_{\nu_{m,n}^\star}(Z_i)$ on $\Omega_{m,n}$, and as $\hat{\mu}_{m,n} \equiv 0$ on $\Omega_{m,n}^c$.
It is shown in Loubes and Pelletier (2008) that $\mathbb{P}(\Omega_{m,n}) \to 1$ as $m\to\infty$ and $n\to\infty$.
Hence for $m$ and $n$ large enough, the AMEM estimate $\hat{\mu}_{m,n}$ may be expressed as in (\ref{eq:cmemonomega}) with high probability, and asymptotically with probability $1$.

\begin{remark}
The construction of the AMEM estimate relies on a discretization of the space $\mathcal{X}$ according to the probability $P_X$. Therefore by varying the support of $P_X$, the practitioner may easily incorporate some a-priori knowledge concerning the support of the solution.
Similarly, the AMEM estimate also depends on the measure $\nu_Z$, which determines the domain of $\Lambda_{\nu_Z}^*$, and so the range of the solution.
\end{remark}

\section{Convergence of the AMEM estimate}
\label{sec:3}
\subsection{Main Result}
\label{subsec:4}

  \hspace{0.36cm} {\bf Assumption 1}
The minimization problem admits at least one solution, i.e., there exists a continuous function $g_0:\mathcal{X}\to \mathrm{co}\, \rm{Supp}\, \nu_Z$ such that
$$\int_{\mathcal{X}} \Phi(x)g_0(x)dP_X(x) \in K_Y .$$

{\bf Assumption 2}
\begin{itemize}
\item[(i)] ${\rm dom}\: \Lambda_{\nu_Z}:=\{s : |\Lambda_{\nu_Z}(s)|<\infty\} = \mathbb{R}$;
\item[(ii)] $\Lambda'_{\nu_Z}$ and $\Lambda''_{\nu_Z}$ are bounded.
\end{itemize}

{\bf Assumption 3} 
The approximating sequence $\Phi_m$ converges to $\Phi$ in $L^2(\mathcal{X},P_X)$. Its rate of convergence is given by
$$ \Vert \Phi_m - \Phi \Vert_{\mathbb{L}^2}= O(\varphi_m^{-1})$$
 
{\bf Assumption 4} $\Lambda_{\nu_Z}$ is a convex function \\

{\bf Assumption 5} For all $m$, the components of $\Phi_m$ are linearly independent\\

{\bf Assumption 6} $\Lambda'_{\nu_Z}$ and $\Lambda''_{\nu_Z}$ are continuous functions.\\
We are now in a position to state our main result.

\begin{theorem}[Convergence of the AMEM estimate]\label{th:main}
Suppose that Assumption~1, Assumption~2, and Assumption~3 hold.
Let $\mu^*$ be the minimizer of the functional
$$I_{\nu_Z}(\mu | P_X) = \int_{\mathcal{X}} \Lambda_{\nu_Z}^*\left(\frac{d\mu}{dP_X}\right)dP_X$$
subject to the constraint $\mu \in S(K_Y) = \{\mu\in\mathcal{M}(\mathcal{X}) : \int_{\mathcal{X}} \Phi(x)d\mu(x) \in K_Y \}.$
\begin{itemize}
\item Then the AMEM estimate $\hat{\mu}_{m,n}$ is defined by 
$$ {\hat\mu_{m,n} = \frac{1}{n} \displaystyle \sum_{i=1}^n \Lambda_{\nu_Z}'(\langle \hat{v}_{m,n},\Phi_m(X_i)\rangle)\delta_{X_i}}  $$
where ${\hat{v}_{m,n}}$ minimizes  on $ \mathbb{R}^k$  
$$ { H_n(\Phi_m,v)=\frac{1}{n} \displaystyle \sum_{i=1}^n \Lambda_{\nu_Z}(\langle v,\Phi_m(X_i)\rangle) - \underset{y \in K_Y}{\inf} \langle v,y \rangle} $$ 
\item Moreover, under Assumption~4,  Assumption~2, and Assumption~3, it converges weakly to $\mu^*$ as $m\to\infty$ and $n\to\infty$. Its rate of convergence is given by 
$$ \Vert\hat\mu_{m,n}-\mu^*\Vert_{VT} = O_P(\varphi_m^{-1})  + O_P \left(\dfrac{1}{\sqrt{n}} \right). $$
\end{itemize}
\end{theorem}

\begin{remark}
Assumption 2-(i) ensures that the function $H(\Phi,v)$ in Theorem~\ref{th:main} attains its minimum at a unique point $v^\star$ belonging to the interior of its domain.
If this assumption is not met, Borwein and Lewis (1993) and Gamboa and Gassiat (1999) have shown that the minimizers of $I_{\nu_Z}(\mu | P_X)$ over $S(K_Y)$ may have a singular part with respect to $P_X$.
\end{remark}

\begin{proof}
The rate of convergence of the AMEM estimate depends both on the discretization $n$ and the convergence of the approximated operator $m$. Hence we consider
\begin{eqnarray*} 
\hat{v}_{m,\infty} & = & \underset{v \in \mathbb{R}^k}{\text{argmin}} \ H(\Phi_m,v)   =   \underset{v \in \mathbb{R}^k}{\text{argmin}} \left\{ \int_\mathcal{X} \Lambda_{\nu_Z}(\langle  \Phi_m(x),v \rangle ) dP_X - \underset{y \in K_Y}{\inf} \langle v,y \rangle \right\}, \\
\hat\mu_{m,n} & = & \frac{1}{n} \displaystyle \sum_{i=1}^n \Lambda_{\nu_Z}'(\langle \hat{v}_{m,n},\Phi_m(.)\rangle)\delta_{X_i},\\
\hat\mu_{m,\infty}\! & = & \Lambda'_{\nu_Z}(\langle \Phi_m(.),\hat{v}_{m,\infty}\rangle) P_X. 
\end{eqnarray*} 
We have the following upper bound
$$\Vert \hat\mu_{m,n} - \mu^* \Vert_{VT} \leq \Vert \hat\mu_{m,n} - \hat\mu_{m,\infty} \Vert_{VT} + \Vert \hat\mu_{m,\infty} - \mu^* \Vert_{VT}, $$
where each term must be tackled separately.\\
First, let us consider $\Vert \hat\mu_{m,n} - \hat\mu_{m,\infty} \Vert_{VT}$.
\begin{eqnarray*}
\Vert \hat\mu_{m,n}\!\!- \hat\mu_{m,\infty} \Vert_{VT}\!\!\!&=&\!\!\!\Vert \dfrac{1}{n} \sum_{i=1}^n \Lambda'_{\nu_Z}(\langle \Phi_m , \hat{v}_{m,n} \rangle)\delta_{X_i} - \Lambda'_{\nu_Z}(\langle \Phi_m , \hat{v}_{m,\infty} \rangle) P_X  \Vert_{VT}\\ 
&\leq&\!\!\!\Vert \dfrac{1}{n} \sum_{i=1}^n \left( \Lambda'_{\nu_Z}(\langle \Phi_m , \hat{v}_{m,n} \rangle) - \Lambda'_{\nu_Z}(\langle \Phi_m ,\hat{v}_{m,\infty} \rangle) \right) \delta_{X_i}\Vert_{VT} \\
& & + \ \Vert \dfrac{1}{n} \sum_{i=1}^n \Lambda'_{\nu_Z}(\langle \Phi_m , \hat{v}_{m,\infty} \rangle)\delta_{X_i} - \Lambda'_{\nu_Z}(\langle \Phi_m , \hat{v}_{m,\infty} \rangle)\! P_X \Vert_{VT}
\end{eqnarray*}
To bound the first term $\Vert \dfrac{1}{n} \displaystyle\sum_{i=1}^n \left( \Lambda'_{\nu_Z}(\langle \Phi_m , \hat{v}_{m,n} \rangle) - \Lambda'_{\nu_Z}(\langle \Phi_m , \hat{v}_{m,\infty} \rangle) \right) \delta_{X_i}\Vert_{VT}$, let $g$ be a bounded measurable function and write
\begin{eqnarray*} & \! \! \! \! \! \!  & \dfrac{1}{n} \sum_{i=1}^n g(X_i) \left( \Lambda'_{\nu_Z}(\langle \Phi_m(X_i) , \hat{v}_{m,n} \rangle) - \Lambda'_{\nu_Z}(\langle \Phi_m(X_i) , \hat{v}_{m,\infty} \rangle) \right) \ \ \ \ \ \ \ \  \ \ \  \\
& \leq & \Vert g \Vert_\infty \Vert \Lambda''_{\nu_Z}\Vert_\infty \dfrac{1}{n} \sum_{i=1}^n \langle \Phi_m(X_i) , \hat{v}_{m,n} - \hat{v}_{m,\infty} \rangle \\
& \leq & \Vert g \Vert_\infty \Vert \Lambda''_{\nu_Z}\Vert_\infty \Vert \hat{v}_{m,n} - \hat{v}_{m,\infty} \Vert \: \dfrac{1}{n} \sum_{i=1}^n \Vert \Phi_m(X_i) \Vert
\end{eqnarray*}
where we have used Cauchy-Schwarz inequality. Since $(\Phi_m)_m$ converges in $\mathbb{L}^2(P_X)$, it is bounded in $\mathbb{L}^2(P_X)$-norm, yelding that  $\frac{1}{n} \sum_{i=1}^n \Vert \Phi_m(X_i) \Vert$ converges almost surely to $ \mathbb{E}\Vert \Phi_m(X) \Vert < \infty$. Hence, there exists  $K_1 > 0$ such that
$$ \Vert \dfrac{1}{n} \sum_{i=1}^n \left( \Lambda'_{\nu_Z}(\langle \Phi_m , \hat{v}_{m,n} \rangle) - \Lambda'_{\nu_Z}(\langle \Phi_m , \hat{v}_{m,\infty} \rangle) \right) \delta_{X_i}\Vert_{VT} \leq K_1  \Vert \hat{v}_{m,n} - \hat{v}_{m,\infty} \Vert  . $$
For the second term, we obtain
$$\Vert \dfrac{1}{n} \displaystyle\sum_{i=1}^n \Lambda'_{\nu_Z}(\langle \Phi_m , \hat{v}_{m,\infty} \rangle)\delta_{X_i} - \Lambda'_{\nu_Z}(\langle \Phi_m , \hat{v}_{m,\infty} \rangle)\! P_X \Vert_{VT} = O_P \left( \dfrac{1}{\sqrt{n}}\right). $$
Hence we get 
$$ \Vert \hat\mu_{m,n} - \hat\mu_{m,\infty} \Vert_{VT} \leq K_1  \Vert \hat{v}_{m,n} - \hat{v}_{m,\infty} \Vert  + O_P \left( \dfrac{1}{\sqrt{n}}\right). $$
The second step is to consider $\Vert \hat\mu_{m,\infty} - \mu^* \Vert_{VT} $ and to follow the same guidelines. So, we get
\begin{eqnarray*}
\Vert \hat\mu_{m,\infty}\! -\! \mu^* \Vert_{VT}\!\!\!&=&\!\!\! \Vert \left( \Lambda'_{\nu_Z}(\langle \Phi_m , \hat{v}_{m,\infty} \rangle) - \Lambda'_{\nu_Z}(\langle \Phi , v^* \rangle) \right)\! P_X \Vert_{VT} \\
&\leq&\!\!\!\Vert \left( \Lambda'_{\nu_Z}(\langle \Phi_m , \hat{v}_{m,\infty} \rangle) - \Lambda'_{\nu_Z}(\langle \Phi_m , v^* \rangle) \right)\! P_X \Vert_{VT}\\
& & \ + \Vert \left( \Lambda'_{\nu_Z}(\langle \Phi_m , v^* \rangle) - \Lambda'_{\nu_Z}(\langle \Phi , v^* \rangle) \right)\! P_X \Vert_{VT} 
\end{eqnarray*}
Fo any bounded measurable function $g$, we can write still using Cauchy-Schwarz inequality that 
\begin{eqnarray*} &\! \! \! \! \! \!  & \int_\mathcal{X} g(x) \left( \Lambda'_{\nu_Z}(\langle \Phi_m(x) , \hat{v}_{m,\infty} \rangle) - \Lambda'_{\nu_Z}(\langle \Phi_m(x) , v^* \rangle) \right) dP_X(x) \ \ \ \ \ \ \  \ \ \  \\
& \leq & \int_\mathcal{X} g(x) \Lambda''_{\nu_Z}(\xi) \langle \Phi_m(x),\hat{v}_{m,\infty} - v^* \rangle dP_X(x) \\
& \leq & \Vert \Lambda''_{\nu_Z} \Vert_\infty \sqrt{\mathbb{E}(g(X))^2 } \sqrt{\mathbb{E}(\Vert \Phi_m(X) \Vert^2)} \:  \Vert \hat{v}_{m,\infty} - v^* \Vert 
\end{eqnarray*}
Hence there exists $K_2 > 0$ such that
$$ \Vert \left( \Lambda'_{\nu_Z}(\langle \Phi_m , \hat{v}_{m,\infty} \rangle) - \Lambda'_{\nu_Z}(\langle \Phi_m , v^* \rangle) \right)\! P_X \Vert_{VT} \leq K_2 \Vert \hat{v}_{m,\infty} - v^* \Vert .$$
Finally, the last term  $\Vert \left( \Lambda'_{\nu_Z}(\langle \Phi_m , v^* \rangle) - \Lambda'_{\nu_Z}(\langle \Phi , v^* \rangle) \right) P_X \Vert_{VT}$ can be bounded. Indeed, for any measurable bounded $g$ 
\begin{eqnarray*} &\! \! \! \! \! \!  & \int_\mathcal{X} g(x) \left( \Lambda'_{\nu_Z}(\langle \Phi_m(x) , v^* \rangle) - \Lambda'_{\nu_Z}(\langle \Phi(x) , v^* \rangle) \right) dP_X(x) \ \ \ \ \ \ \  \ \ \ \ \ \ \ \ \\
 & = & \int_\mathcal{X} g(x) \Lambda''_{\nu_Z}(\xi_x) \langle \Phi_m(x) - \Phi(x),v^* \rangle dP_X(x) \\
 & \leq & \int_\mathcal{X} g(x) \Lambda''_{\nu_Z}(\xi_x) \Vert \Phi_m(x) - \Phi(x) \Vert \Vert v^* \Vert dP_X(x) \\
& \leq & \Vert v^* \Vert \Vert \Lambda''_{\nu_Z} \Vert_\infty \sqrt{\mathbb{E}(g(X))^2 } \sqrt{\mathbb{E}(\Vert \Phi_m(X )- \Phi(X) \Vert^2)} 
\end{eqnarray*}
Hence there exists  $K_3 > 0$ such that
$$ \Vert \left( \Lambda'_{\nu_Z}(\langle \Phi_m , v^* \rangle) - \Lambda'_{\nu_Z}(\langle \Phi , v^* \rangle) \right) P_X \Vert_{VT} \leq K_3 \Vert \Phi_m - \Phi \Vert_{\mathbb{L}^2} $$
We finally obtain the following bound
$$ \Vert \hat\mu_{m,n} - \mu^* \Vert_{VT} \leq  K_1  \Vert \hat{v}_{m,n} - \hat{v}_{m,\infty} \Vert +  K_2 \Vert \hat{v}_{m,\infty} - v^* \Vert  + K_3 \Vert \Phi_m - \Phi \Vert_{\mathbb{L}^2} + O_P \left( \dfrac{1}{\sqrt{n}}\right) $$
Using Lemmas \ref{lem1} and \ref{lem2}, we obtain that 
$$ \Vert \hat{v}_{m,n} - \hat{v}_{m,\infty} \Vert = O_P \left( \frac{1}{\sqrt{n}} \right) $$
$$\Vert \hat{v}_{m,\infty} - v^* \Vert = O_P (\varphi_m^{-1})$$
Finally, we get
$$ \Vert\hat\mu_{m,n}-\mu^*\Vert_{VT} = O_P(\varphi_m^{-1})  + O_P \left(\frac{1}{\sqrt{n}} \right), $$ which proves the result. \qed
\end{proof}

\subsection{Application to remote sensing}
\label{subsec:5}
In remote sensing of aerosol vertical profiles, one wishes to recover the concentration of aerosol particules from noisy observations of the radiance field (i.e., a radiometric quantity), in several spectral bands (see e.g. Gabella et al, 1997; Gabella, Kisselev and Perona, 1999).
More specifically, at a given level of modeling, the noisy observation $y^{obs}$ may be expressed as
\begin{equation}
\label{eq:aero}
y^{obs} = \int_{\mathcal{X}}\Phi(x;t^{obs})d\mu_X(x) + \varepsilon,
\end{equation}
where $\Phi:\mathcal{X}\times\mathcal{T}\to\mathbb{R}^k$ is a given operator, and where $t^{obs}$ is a vector of angular parameters observed simultaneously with $y^{obs}$.
The aerosol vertical profile is a function of the altitude $x$ and is associated with the measure $\mu_X$ to be recovered, i.e., the aerosol vertical profile is the Radon-Nykodim derivative of $\mu_X$ with respect to a given reference measure (e.g., the Lebesgue measure on $\mathbb{R}$).
The analytical expression of $\Phi$ is fairly complex as it sums up several models at the microphysical scale, so that basically $\Phi$ is available in the form of a computer code.
So this problem motivates the introduction of an efficient numerical procedure for recovering the unknwon $\mu_X$ from $y^{obs}$ and arbitrary $t^{obs}$.\\

More generally, the remote sensing of the aerosol vertical profile is in the form of an inverse problem where some of the inputs (namely $t^{obs}$) are observed simultaneously with the noisy output $y^{obs}$.
Suppose that random points $X_1,\dots,X_n$ of $\mathcal{X}$ have been generated.
Then, applying the maximum entropy approach would require the evaluations of $\Phi(X_i,t^{obs})$ each time $t^{obs}$ is observed.
If one wishes to process a large number of observations, say $(y^{obs}_i,t^{obs}_i)$, for different values $t^{obs}_i$, the computational cost may become prohibitive.
So we propose to replace $\Phi$ by an approximation $\Phi_m$, the evaluation of which is faster in execution.
To this aim, suppose first that $\mathcal{T}$ is a subset of $\mathbb{R}^p$.
Let $T_1,...,T_m$ be random points of $\mathcal{T}$, independent of $X_1,\dots,X_n$, and drawn from some probability measure $\mu_{T}$ on $\mathcal{T}$ admitting a density $f_T$ with respect to the Lebesgue measure on $\mathbb{R}^p$ such that $f_T(t)>0$ for all $t\in\mathcal{T}$.
Next, consider the operator
$$\Phi_m(x,t) = \frac{1}{f_T(t)}\frac{1}{m}\sum_{i=1}^mK_{h_m}(t-T_i)\Phi(x,T_i),$$
where $K_{h_m}(.)$ is a symetric kernel on $\mathcal{T}$ of smoothing sequence $h_n$.
It is a classical result to prove that $\Phi_m$ converges to $\Phi$ in quadratic norm provided  $h_m$ tends to $0$ at a suitable rate, which ensures that Assumption 3 of Theorem~\ref{th:main} is satisfied.
Since the $T_i$'s are independent from the $X_i$, one may see that Theorem~\ref{th:main} applies, and so the solution to the approximate inverse problem
$$y^{obs}=\int_{\mathcal{X}}\Phi_m(x;t^{obs})d\mu_X(x) + \varepsilon,$$
will converge to the solution to the original inverse problem in Eq.~\ref{eq:aero}.
In terms of computational complexity, the advantage of this approach is that the construction of the AMEM estimate requires, for each new observation $(y^{obs},t^{obs})$, the evaluation of the $m$ kernels at $t^{obs}$, i.e., $K_{h_m}(t^{obs}-T_i)$, the $m\times n$ ouputs $\Phi(X_i,T_j)$ for $i=1,\dots,n$ and $j=1,\dots,m$ having evaluated once and for all.

\subsection{Application to deconvolution type problems in optical nanoscopy} \label{subsec:6}
Following the framework defined in \cite{Hohage}, the number of photons counted can be expressed using a convolution of $p(x-y,y)$ the probability of recording a photon emission at point $y$ when illuminating point $x$, with $d\mu(y)=f(y)dy$ the measure of the fluorescent markers.
$$g(x)= \int p(x-y,x) f(y) dy. $$
Here $p(x-y,y)=p(x,y,\phi(x))$.
Reconstruction of $\mu$ can be achieved using AMEM technics.

\section{Tecnical Lemmas}
\label{sec:4}
Recall the following definitions 
\begin{eqnarray*} 
\hat{v}_{m,\infty} & = & \underset{v \in \mathbb{R}^k}{\text{argmin}} \ H(\Phi_m,v)   =   \underset{v \in \mathbb{R}^k}{\text{argmin}} \left\{ \int_\mathcal{X} \Lambda_{\nu_Z}(\langle  \Phi_m(x),v \rangle ) dP_X - \underset{y \in K_Y}{\inf} \langle v,y \rangle \right\}, \\
\hat{v}_{m,n} & = & \underset{v \in \mathbb{R}^k}{\text{argmin}} \ H_n(\Phi_m,v)= \underset{v \in \mathbb{R}^k}{\text{argmin}}  \left\{ \frac{1}{n} \displaystyle \sum_{i=1}^n \Lambda_{\nu_Z}(\langle v,\Phi_m(X_i)\rangle) - \underset{y \in K_Y}{\inf} \langle v,y \rangle \right\}, \\
v^* & = &  \underset{v \in \mathbb{R}^k}{\text{argmin}} H(\Phi,v) = \underset{v \in \mathbb{R}^k}{\text{argmin}} \left\{ \int_\mathcal{X} \Lambda_{\nu_Z}(\langle  \Phi(x),v \rangle )dP_X(x) - \underset{y \in K_Y}{\inf} \langle v,y\rangle  \right\}
\end{eqnarray*} 
\begin{lemma}[Uniform convergence at a given approximation level $m$] \label{lem1}
For all $m$, we get
$$ \Vert \hat{v}_{m,n} - \hat{v}_{m,\infty} \Vert = O_P \left( \frac{1}{\sqrt{n}} \right) $$
\end{lemma}
\begin{proof}
$\hat{v}_{m,n}$ is defined as the minimizer of an empirical constrast function  $H_n(\Phi_m,.)$.  Indeed, set $$h_m(v,x) = \Lambda_{\nu_Z}( \langle \Phi_m(x),v \rangle ) - \underset{y \in K_Y}{\inf } \langle v,y \rangle ,$$ hence $$H(\Phi_m,v) = P_X h_m(v,.).$$
Using classical theorem from the theory of M-estimation, we get the convergence in probability of $\hat{v}_{m,n}$ towards $\hat{v}_{m,\infty}$ provided that the contrast  converges uniformly over every compact set of $\mathbb{R}^k$ towards  $H(\Phi_m,.)$ when $n \rightarrow \infty$. More precisely Corollary 5.53 in van der Vaart
 (1998) states that if we consider $x \mapsto h_m (v,x)$  a measurable function and  $ \overset{.}{h_m} $ a function in $L^2(P)$, such that for all $v_1$ and $v_2$ in a neighbourhood of  $v^*$
$$ \vert h_m(v_1,x) - h_m(v_2, x) \vert \leq  \overset{.}{h_m} (x) \Vert v_1 - v_2 \Vert .$$
Moreover if $v \mapsto P h_m(v,.)$ has a Taylor expansion of order at least 2 around its unique minimum $v^*$ and if the Hessian matrix at this point is positive, hence provided  $\; \mathbb{P}_n h_m({\hat v_n},.) \leq \; \mathbb{P}_n h_m(v^*,) + \; O_P(n^{-1})$ then
$$ \sqrt{n}(\hat v_n - v^*) = O_P(1).$$
We want to apply this result to our problem. Let  $\eta$ be an un upper bound for $\|\varepsilon\|$, we set 
$ h_m(v,x) = \Lambda_{\nu_Z}( \langle \Phi_m(x),v \rangle ) - \langle v, y^{obs} \rangle - \underset{ \Vert y - y^{obs} \Vert \leq \eta }{\inf } \langle v,y - y_{obs} \rangle. $ Now note that $z \mapsto \langle v, z \rangle $ reaches its minimum on $ \mathcal{B}(0,\eta)$ at the point $- \eta \dfrac{v}{\Vert v \Vert}$, so
$$h_m(v,x) = \Lambda_{\nu_Z}( \langle \Phi_m(x),v \rangle ) - \langle v, y^{obs} \rangle + \eta \Vert v \Vert $$
For all $v_1$, $v_2 \in \mathbb{R}^k$, we have 
\begin{eqnarray*}
& & \vert h_m(v_1,x) - h_m(v_2,x) \vert \\
 & = & \vert \Lambda_{\nu_Z}( \langle \Phi_m(x),v_1 \rangle ) - \underset{y \in K_Y}{\inf } \langle v_1,y \rangle - \Lambda_{\nu_Z}( \langle \Phi_m(x),v_2 \rangle ) + \underset{y \in K_Y}{\inf } \langle v_2,y \rangle \vert \\
 & \leq & \vert \Lambda_{\nu_Z}( \langle \Phi_m(x),v_1 \rangle ) - \Lambda_{\nu_Z}( \langle \Phi_m(x),v_2 \rangle ) \vert + \vert \underset{y \in K_Y}{\inf } \langle v_2,y \rangle - \underset{y \in K_Y}{\inf } \langle v_1,y \rangle \vert \\
 & \leq & \vert \Lambda_{\nu_Z}( \langle \Phi_m(x),v_1 \rangle ) - \Lambda_{\nu_Z}( \langle \Phi_m(x),v_2 \rangle ) \vert + \vert \langle v_2 - v_1, y^{obs} \rangle - \eta (\Vert v_2 \Vert - \Vert v_1 \Vert) \vert \\
 & \leq & \left( \Vert \Lambda_{\nu_Z}' \Vert_\infty \Vert \Phi_m(x) \Vert + \Vert y^{obs} \Vert + \eta \right) \Vert v_1 - v_2 \Vert 
\end{eqnarray*}
Define $\overset{.}{h_m} : x \mapsto \Vert \Lambda_{\nu_Z}' \Vert_\infty \Vert \Phi_m(x) \Vert + \Vert y^{obs} \Vert + \eta $. Since $(\Phi_m)_m$ is bounded in $\mathbb{L}^2(P_X)$,  $(\overset{.}{h_m})_m$ is in $L^2(P_X)$ uniformly with respect to $m$, which entails that
\begin{equation} \label{H1}
\exists K, \forall m, \int_\mathcal{X} \overset{.}{h_m}^2  dP_X < K
\end{equation}
Hence the function  $\overset{.}{h_m}$ satisifes the first condition 
$$ \vert h_m(v_1,x) - h_m(v_2,x) \vert \leq \overset{.}{h_m} (x) \Vert v_1 - v_2 \Vert $$
Now, consider $H(\Phi_m,.)$ Let $V_{m,v}$ be the Hessian matrix of $H(\Phi_m,.)$ at point $v$. We need to prove that $V_{m,\hat{v}_{m,\infty}}$ is non negative. Let $\partial_i$ be the derivative with respect to the $i^{\rm th}$ component. Set $v \neq 0$, we have
\begin{eqnarray*}
{V_{m,v}}^{ij}(v) = \partial_i\partial_jH(  \Phi_m,v) & = & \int_{\mathcal{X}}\partial_i\partial_j h_m(v,x) dP_X\\
& = & \int_{\mathcal{X}}\Phi_m^i(x)\Phi_m^j(x) \Lambda_{\nu_Z}''( \langle \Phi_m(x),v \rangle ) dP_X + \eta \; \partial_i\partial_j N(v)
\end{eqnarray*}
where let $N$ be $N:v \mapsto \Vert v \Vert$.\\
Hence the Hessian matrix $V_{m,\hat{v}_{m,\infty}}$ of $H(\Phi_m,.)$ at point $\hat{v}_{m,\infty}$ can be split into the sum ot the following matrices
\begin{eqnarray*}
(M_1)_{ij} & = & \int_{\mathcal{X}}\Phi_m^i(x)\Phi_m^j(x) \Lambda_{\nu_Z}''( \langle \Phi_m(x),\hat{v}_{m,\infty} \rangle ) dP_X, \\
(M_2)_{ij} & = & \partial_i\partial_j N(\hat{v}_{m,\infty}).
\end{eqnarray*}
Under Assumptions  (A3) and (A5), $\Lambda_{\nu_Z}''$ is positive and belongs to $L_1(P_X)$ since it is bounded. So we can define $\int_{\mathcal{X}}\Phi_m^i(x)\Phi_m^j(x) \Lambda_{\nu_Z}''( \langle \Phi_m(x),\hat{v}_{m,\infty} \rangle ) dP_X$ as the scalar product of $ \Phi_m^i$ and $ \Phi_m^j$  in the space $\mathbb{L}^2(\Lambda_{\nu_Z}''( \langle \Phi_m(.),\hat{v}_{m,\infty} \rangle P_X)$.\\
$M_1$ is a Gram matrix, hence using (A6) it is a non negative matrix.\\
$M_2$ can be computed as follows. For all  $v \in \mathbb{R}^k/ \lbrace 0 \rbrace$, we have
\begin{eqnarray*}
N(v) & = & \sqrt{\textstyle\sum_{i=1}^k v_i^2} \\
\partial_iN(v) & = & \dfrac{v_i}{\Vert v \Vert}\\
\partial_i \partial_j N(v) & = &  \left\lbrace  - \dfrac{v_iv_j}{\Vert v \Vert ^3} \: \: \: \: \: \: \  \text{si } i \neq j  \atop  \dfrac{\Vert v \Vert^2 - v_i^2}{\Vert v \Vert ^3} \: \:   \text{si } i=j  \right.   
\end{eqnarray*} 
Hence for all $a \in \mathbb{R}^k$, we can write
\begin{eqnarray*}
& & a^TM_2a \\
& = & \sum_{1 \leq i,j \leq k} \partial_i\partial_jN(\hat{v}_{m,\infty})a_ia_j\\
        & = & \sum_{i=1}^k \frac{\Vert \hat{v}_{m,\infty} \Vert^2 - {\hat{v}_{m,\infty,i}}^2}{\Vert \hat{v}_{m,\infty} \Vert ^3}a_i ^2 - \sum_{i \neq j} \frac{\hat{v}_{m,\infty,i} \hat{v}_{m,\infty,j}}{\Vert \hat{v}_{m,\infty} \Vert ^3}a_ia_j \\
        & = & \dfrac{1}{\Vert \hat{v}_{m,\infty} \Vert^3} \left(  \Vert \hat{v}_{m,\infty} \Vert^2 \sum_{i=1}^k a_i^2 - \sum_{i=1}^k a_i^2 {\hat{v}_{m,\infty,i}}^2 - \sum_{1 \leq i,j \leq k} a_i \hat{v}_{m,\infty,i} a_j \hat{v}_{m,\infty,j} + \sum_{i=1}^ka_i^2 {\hat{v}_{m,\infty,i}}^2 \right)  \\
        & = & \dfrac{1}{\Vert \hat{v}_{m,\infty} \Vert^3} \left( \Vert \hat{v}_{m,\infty} \Vert^2 \Vert a \Vert ^2 - \sum_{1 \leq i,j \leq k} a_i \hat{v}_{m,\infty,i} a_j \hat{v}_{m,\infty,j} \right) \\
        & = & \dfrac{1}{\Vert \hat{v}_{m,\infty} \Vert^3} \left(  \Vert \hat{v}_{m,\infty} \Vert ^2 \Vert a \Vert ^2 - \langle a,\hat{v}_{m,\infty} \rangle ^2 \right) \geq  0 \: \:  \text{using  Cauchy-Schwarz's inequality.}\\
\end{eqnarray*}
So $M_2$ is clearly non negative, hence  $V_{m,\hat{v}_{m,\infty}} = M_1 + \eta M_2$ is also non negative. Finally we conclude that $H(\Phi_m,.)$ undergoes the assumptions of Theorem  5.1. \qed
\end{proof}
\begin{lemma} \label{lem2}
$$\Vert \hat{v}_{m,\infty} - v^* \Vert = O_P (\varphi_m^{-1})$$
\end{lemma}
\begin{proof}
First write,
\begin{eqnarray*}
\vert H(\Phi_m,v) - H(\Phi,v) \vert & = & \vert \int_\mathcal{X} \Lambda_{\nu_Z}( \langle \Phi_m(x),v \rangle) - \Lambda_{\nu_Z}( \langle \Phi(x),v \rangle ) dP_X(x) \vert \\
& \leq & \Vert \Lambda'_{\nu_Z} \Vert_\infty \Vert v \Vert \Vert \Phi_m - \Phi \Vert_{\mathbb{L}^2},
\end{eqnarray*} which implies uniform convergence over every compact set of $ H(\Phi_m,.) $ towards $H(\Phi,.)$ when $m \rightarrow \infty$, yelding   that $\hat{v}_{m,\infty} \rightarrow v^*$ in probability. To compute the rate of convergence, we use Lemma~\ref{Lemme52}. As previously we can show that the Hessian matrix of $H( \phi,.)$ at point $v^*$ is positive.
We need to prove uniform convergence of  $\nabla H( \phi_m,.  )$ towards $\nabla H( \phi,.  )$. For this, write 
\begin{eqnarray*}
& & \partial_i \left[ H( \phi_m,.  ) - H( \phi,.  ) \right] (v) \\
  & = & \! \! \! \int_\mathcal{X} \Phi^i_m(x) \Lambda'_{\nu_Z}( \langle \Phi_m(x),v \rangle) - \Phi^i(x) \Lambda'_{\nu_Z}( \langle \Phi(x),v \rangle ) dP_X(x) \\
& = & \! \! \! \int_\mathcal{X} (\Phi^i_m - \Phi^i)(x) \Lambda'_{\nu_Z}( \langle \Phi_m(x),v \rangle) - \Phi^i(x) \Lambda''_{\nu_Z}( \xi) \langle (\Phi - \Phi_m)(x),v \rangle  dP_X(x) \\
& \leq & \! \! \! \Vert \Phi^i - \Phi^i_m\Vert_{\mathbb{L}^2} \Vert \Lambda'_{\nu_Z}\Vert_\infty + \Vert \Phi^i \Vert_{\mathbb{L}^2} \Vert \Lambda''_{\nu_Z} \Vert_\infty  \Vert \Phi - \Phi_m \Vert_{\mathbb{L}^2} \Vert v \Vert  
\end{eqnarray*}
using again Cauchy-Schwarz's inequality. Finally we obtain 
$$ \Vert  \nabla \left( H( \phi_m,.  ) - H(\Phi,.)\right) (v) \Vert \leq (C_1 + C_2 \Vert v \Vert ) \  \Vert \Phi - \Phi_m \Vert_{\mathbb{L}^2} $$
for positive constants $C_1$ and $C_2$. For any compact neighbourhood of $v^*$, $\mathcal{S}$, the function $v \mapsto \Vert  \nabla \left( H( \phi_m,.  ) - H(\Phi,.)\right) (v) \Vert$ converges uniformly to $0$. But for $m$ large enough, $\hat{v}_{m,\infty} \in \mathcal{S}$ almost surely. Using 2. in Lemma~\ref{Lemme52} with the function  $v \mapsto \Vert  \nabla \left( H( \phi_m,.  ) - H(\Phi,.)\right) (v) \Vert {\bf 1}_{\mathcal{S}}(v)$ converging uniformly to  $0$, implies that   
$$\Vert \hat{v}_{m,\infty} - v^* \Vert = O_P (\varphi_m^{-1}). \qed $$
\end{proof}

\begin{lemma} \label{Lemme52} 
Let $f$ be defined on $ \mathcal{S} \subset \mathbb{R}^d \rightarrow \mathbb{R}$, which reaches a unique minimum at point $\theta_0$. Let  $(f_n)_n$ be a sequence of continuous functions which converges uniformly towards $f$. Let $\hat{\theta}_n = {\rm arg}\min f_n$. If $f$ is twice differentiable on a neighbourhood of $\theta_0$ and provided its Hessian matrix  $V_{\theta_0}$ is non negative, hence we get
\begin{enumerate} 
\item there exists a positive constant $C$ such that
$$ \Vert \hat\theta_n- \theta_0 \Vert \leq C \sqrt{\Vert f - f_n \Vert_\infty} $$ 
\item Moreover if $\theta \mapsto V_\theta$ is continuous in a neighbourhood of $\theta_0$ and $\Vert  \nabla f_n(.) \Vert$ uniformly converges towards  $\Vert  \nabla f(.) \Vert$, hence there exists a constant  $C'$ such that 
$$ \Vert \hat\theta_n- \theta_0 \Vert \leq C' \Vert  \nabla( f - f_n)\Vert_\infty $$
with $\Vert g \Vert_\infty = \ \underset{x \in \mathcal{S}}{\sup} \ \Vert g(x) \Vert$
\end{enumerate}
\end{lemma}

\begin{proof}
The proof of this classical result in optimization relies on easy convex analysis tricks. For sake of completeness, we recall here the main guidelines.\\
1. There are non negative constants $C_1$ et $\delta_0$ such that  
$$\forall \; 0 < \delta \leq \delta_0, \underset{d(\theta , \theta_0) > \delta}{\inf} f(\theta) - f(\theta_0) > C_1\delta^2$$  
Set $ \Vert f_n - f \Vert_\infty = \varepsilon_n$. For $0 < \delta_1 < \delta_0$, let $n$ be chosen such that $2\varepsilon_n \leq C_1\delta_1^2$. Hence
$$ \underset{d(\theta,\theta_0)> \delta_1}{\inf } f_n(\theta) \geq \underset{d(\theta,\theta_0)> \delta_1}{\inf } f(\theta) - \varepsilon_n > f(\theta_0) + \varepsilon_n \geq f_n (\theta_0)$$
Finally  $ f_n(\theta_0) < \underset{d(\theta,\theta_0)> \delta_1}{\inf } f_n(\theta) $ \: $ \Longrightarrow \hat\theta_n \in \lbrace \theta: d(\theta,\theta_0) \leq \delta_1 \rbrace$, which enables to conclude setting  $C=\sqrt{\frac{2}{C_1}}$.\\

2. We prove the result for $d = 1$, which can be easily extended for all $d$. Using Taylor-Lagrange expansion, there exists $\tilde{\theta}_n \in \ ]\hat\theta_n, \theta_0 [$ such that 
$$f'(\theta_0) = 0 = f'(\hat\theta_n) + (\theta_0- \hat\theta_n)f''( \tilde{\theta}_n).  $$
Remind that $f''( \tilde{\theta}_n) \underset{n \rightarrow \infty}{\longrightarrow} f''( \theta_0) > 0$. So, for $n$ large enough there exits  $C' > 0$ such that
$$ \vert \theta_0- \hat\theta_n \vert = \dfrac{\vert f'(\hat\theta_n) - f'(\theta_0) \vert}{\vert f''( \tilde{\theta}_n) \vert} \leq C' \Vert f' - f'_n \Vert_\infty , $$ which ends the proof. \qed
\end{proof}

\input{referenc}

\end{document}

%% file: macros.tex
\newcommand{\DOIFPDF}[2]{\ifx\pdfoutput\undefined #2\else#1\fi}

\newcommand{\PDFABLE}[2]{%
\newif\ifpdf%
\ifx\pdfoutput\undefined\pdffalse%
\else\pdftrue\pdfoutput=1\pdfcompresslevel=9\fi%
\ifpdf%
 \usepackage#1
 \usepackage[pdftex,%
             a4paper,%
             colorlinks,%
             citecolor=blue,%
             pagebackref,%
             plainpages=false]{hyperref}%
\else%
 \usepackage#2
 \usepackage{url}
\fi}

\makeatletter
\newcommand{\@THMSTYLES}{%
  \newtheoremstyle{bodyrm}
  {3pt}
  {3pt}
  {}
  {}
  {\bfseries\sffamily}
  {.}
  { }
  {}
  \newtheoremstyle{bodyit}
  {3pt}
  {3pt}
  {\itshape}
  {}
  {\bfseries\sffamily}
  {.}
  { }
  {}
}
\newcommand{\THMEN}{%
  \@THMSTYLES
  \theoremstyle{bodyit}
  \newtheorem{thm}{Theorem}[section]%
  \newtheorem{cor}[thm]{Corollary}%
  \newtheorem{prop}[thm]{Proposition}%
  \newtheorem{lem}[thm]{Lemma}%
  \theoremstyle{bodyrm}%
  \newtheorem{defi}[thm]{Definition}%
  \newtheorem{xpl}[thm]{Example}%
  \newtheorem{exo}[thm]{Exercise}%
  \newtheorem{hyp}[thm]{Hypothesis}%
  \newtheorem{eur}[thm]{Heuristics}%
  \newtheorem{pro}[thm]{Problem}%
  \newtheorem{rem}[thm]{Remark}%
  \newtheorem{prp}[thm]{Property}%
}
\newcommand{\THMFR}{%
  \@THMSTYLES
  \theoremstyle{bodyit}
  \theoremstyle{bodyrm}%
}
\makeatother

\makeatletter
\newcommand{\SMALLSECS}{%
 \renewcommand{\section}{\@startsection%
  {section}
  {1}
  {0em}
  {\baselineskip}
  {0.5\baselineskip}
  {\normalfont\large\bfseries}}
 \renewcommand{\subsection}{\@startsection%
  {subsection}
  {2}
  {0em}
  {\baselineskip}
  {0.25\baselineskip}
  {\normalfont\bfseries}}
}
\makeatother

\makeatletter
\providecommand{\timenow}{\@tempcnta\time
\@tempcntb\@tempcnta
\divide\@tempcntb60
\ifnum10>\@tempcntb0\fi\number\@tempcntb
\multiply\@tempcntb60
\advance\@tempcnta-\@tempcntb
:\ifnum10>\@tempcnta0\fi\number\@tempcnta}
\makeatother

\makeatletter
\newcommand{\versiondetravail}{%
 \renewcommand{\@evenfoot}{%
 \hfil{\tiny\texttt{%
   Version préliminaire, compilée le \today{} à \timenow.}\hfill}}%
 \renewcommand{\@oddfoot}{\@evenfoot}%
}
\makeatother

\newcommand{\Det}[1]{\mathrm{Det}\,}

\renewcommand{\leq}{\leqslant}
\renewcommand{\geq}{\geqslant}



%% file: referenc.tex
%
%
%